\documentclass[11pt,reqno]{amsart}

\numberwithin{equation}{section}

\usepackage{amssymb}
\usepackage{amsmath}
\usepackage{amsbsy}
\usepackage{amscd}
\usepackage{amsthm}
\usepackage{amsfonts}
\usepackage{enumerate}

\usepackage{color}

\usepackage{ifpdf}
\ifpdf \usepackage[pdftex,pdfstartview=FitH,pdfpagemode=none,colorlinks,bookmarks,linkcolor=blue]{hyperref} \else  \usepackage[hypertex]{hyperref} \fi

\usepackage[nobysame,abbrev,alphabetic]{amsrefs}
\newcommand{\hide}[1]{}

\newtheorem{theorem}{Theorem}[section]

\newtheorem{lemma}[theorem]{Lemma}
\newtheorem{corollary}[theorem]{Corollary}

\newtheorem{conjecture}[theorem]{Conjecture}

\newtheorem{proposition}[theorem]{Proposition}

\newtheorem{remark}[theorem]{Remark}
\theoremstyle{definition}

\newcommand{\bC}{\mathbb{C}}
\newcommand{\bD}{\mathbb{D}}
\newcommand{\bE}{\mathop{\mathbb{E}}}

\newcommand{\bR}{\mathbb{R}}
\newcommand{\bZ}{\mathbb{Z}}

\newcommand{\bN}{\mathbb{N}}

\newcommand{\bT}{\mathbb{T}}

\newcommand{\hh}{{\hat{h}}}

\newcommand{\rmm}{{\mathrm{m}}}

\newcommand{\supp}{\operatorname{supp}}
\newcommand{\re}{\operatorname{Re}}

\newcommand{\onto}{\xymatrix{\ar@{>>}[r]&}}
\newcommand{\da}[4]{\xymatrix{#1 \ar@<.5ex>[r]^{#2} \ar@<-.5ex>[r]_{#3} & #4}}

\newcounter{subconst}[subsection]

\newcounter{const}

\newcounter{CONST}

\begin{document}

\title{Mobius disjointness for analytic skew products}
\author[Z. Wang]{Zhiren Wang}
\address{\newline Pennsylvania State University, University Park, PA 16802, USA\newline \rm zhirenw@psu.edu}
\setcounter{page}{1}
\begin{abstract}
We show that the M\"obius function is disjoint to every analytic skew product dynamical system on $\bT^2$ over a rotation of the circle.
\end{abstract}
\maketitle
{\small\tableofcontents}

\section{Introduction}

Let $h$ be a continuous map from the circle $\bT^1=\bR/\bZ$ to itself. For $\alpha\in[0,1)$ we consider the skew transform on $\bT^2$ given by  \begin{equation}\label{DynaEq}T(x,y)=(x+\alpha, y+h(x)).\end{equation}

We show the following M\"obius disjointness statement:

\begin{theorem}\label{Main} Suppose $h:\bT^1\mapsto\bT^1$ is analytic and $T$ is as above. Then
\begin{equation}\label{MainEq}\frac1N\sum_{n\leq N}\mu(n)f(T^n(x_0,y_0))\rightarrow 0,\ \forall(x_0,y_0)\in\bT^2,\ \forall f\in C^0(\bT^2).\end{equation}\end{theorem}

We emphasize that there is no assumption on $\alpha$.

The M\"obius function is defined by $\mu(n)=(-1)^k$ if $n$ is the product of $k$ distinct primes and $\mu(n)=0$ otherwise. An important theme in dynamical system and number theory during the recent years is the randomness of $\mu$, which is characertized by Sarnak's Mobius Disjointness Conjecture:

\begin{conjecture}\label{SarnakConj}\cite{S09} If $X$ is a compact metric space, $T:X\mapsto X$ is a continuous map with zero topological entropy, then $\frac1N\sum_{n\leq N}\mu(n)f\big(T^n(x)\big)\to 0$ for all continuous functions $f$ on $X$ and all $x\in X$.\end{conjecture}

An important feature of the conjecture is that disjointness is expected to hold for every $x$ instead of for almost every $x$ with respect to some $T$-invariant probability measure.

The case where $X$ is finite is equivalent to the prime number theorem in arithmetic progressions. And the case where $T$ is a rotation on the circle is Davenport's theorem \cite{D37} that $\sum_{n\leq N}\mu(n)e(\alpha n)=o(N)$ uniformly for all $\alpha$.

Many special cases of Sarnak's Conjecture have been established. To list a few: \cites{MR10, G12, GT12, B13a, B13b, BSZ13, KL13, ELD14, MMR14, P15}. The majority of these results rely on Vinogradov's bilinear method, or one of its newer variants: either Vaughan's identity \cite{V77} or the Bourgain-Sarnak-Ziegler criterion \cite{BSZ13}.

It should be remarked that in all the above cases, the dynamical system is regular in the sense that for every point $x$, $\frac1N\sum_{n\leq N}\delta_{T^n(x)}$ converges to some $T$-invariant probability measure on $X$ in weak-$^*$ topology.

Skew products on $\bT^2$ provide the simplest examples of irregular dynamics. In the case of $C^{1+\epsilon}$ skew products, for Diophantine rotation numbers $\alpha$, which have full Lebesgue measure, the dynamics is topologically conjugate to an affine transformation, hence is regular and Conjecture \ref{SarnakConj} holds.  Ku\l{}aga-Przymus and Lema\'nczyk showed in \cite{KL13} that this indeed holds for a much larger topologically generic set of $\alpha$. However, Furstenberg constructed in \cite{F61} a counterexample showing that for some $\alpha$ and some analytic function $h$, the dynamics of $T$ is not regular.

In  \cite{LS15}, Liu and Sarnak established Conjecture \ref{SarnakConj} for a class of analytic skew products of the form \eqref{DynaEq}. This was a remarkable achievement because their result covers a topological conjugate of Furstenberg's counterexample, and therefore it was the first time that Conjecture \ref{SarnakConj} was proved for an irregular dynamical system. The hypothesis in \cite{LS15} was that the Fourier coefficents $h$ decay with exponential bounds both from above and from below, i.e., $e^{-\tau_1|m|}<|\hh(m)|<e^{-\tau_2|m|}$ for some $\tau_1,\tau_2>0$. The upper bound corresponds to analyticity, while the lower bound was posed in order to get extra cancellation using van der Corput's method of exponential sums.

The main result in this paper, Theorem \ref{Main}, proves Conjecture \ref{SarnakConj} for every analytic skew product on $\bT^2$ without assuming a lower bound for $|\hh(m)|$.

Our proof is a mixture of dynamical and number-theoretical arguments. It is based on a dynamical dichotomy due to Furstenberg \cite{F61}, and uses the Bourgain-Sarnak-Ziegler criterion from \cite{BSZ13}. The main new ingredient is the recent estimate by Matom\"aki-Radziwi\l\l-Tao \cite{MRT15} on averages of multiplicative functions in short intervals.

In \cite{F61}, it was shown that $T$ either is uniquely ergodic or preserves an invariant measure that is a finite extension of the Lebesgue measure in the $x$ coordinate. We apply this to a two-parameter family of maps of the form $$(x,y)\mapsto \Big(x+\alpha, y+\sum_{l=0}^{p_1-1}h(x_0+p_1x+l\alpha)-\sum_{l=0}^{p_2-1}h(x_0+p_2x+l\alpha)\Big)$$ where $(p_1, p_2)$ are pairs of distinct primes.

If unique ergodicity holds for all pairs of primes, it will follow that there is no correlation between the sequences $f\big(T^{p_1n}(x_0,y_0)\big)$ and $f\big(T^{p_2n}(x_0,y_0)\big)$  and the Bourgain-Sarnak-Ziegler criterion allows to conclude \eqref{MainEq}.

In the case where unique ergodicity fails for at least one pair $(p_1, p_2)$,  using Fourier series, we are able to show, after several steps of dynamical reductions, that in order for \eqref{MainEq} to fail, the sequence $n\mapsto f(T^n(x_0,y_0))$ must be well approximated by a periodic sequence of period $Q$ in every interval of length $l$. The analyticity assumption allows to make $l$ exponentially large in terms of $Q$, while $Q$ can be arbitrarily large itself. Then we decompose, in each interval of length $l$, the periodic sequence into a linear combination of Dirichlet characters $\chi$ of modulus $Q$. This reduces the problem to controlling the average of the multiplicative function $\mu(n)\chi(n)$ on a typical interval of length $l$.

Matom\"aki, Radziw\l\l{} and Tao proved in \cite{MRT15} that for a non-pretentious multiplicative function $\nu$ and $X\gg l\gg 1$, on most intervals $I\subset[X, 2X)$ of length $l$ the average of $\nu$ is very small. The function $\mu\chi$ is known to be non-pretentious, so this theorem applies in our setting. Moreover, the condition that $l$ is exponentially large compared to $Q$ assures that the bound does not explode when different $\mu\chi$'s are combined and gives the desired disjointness.

In fact, our method easily extends to all $h$ satisfying $|\hh(m)|\ll e^{-\tau|m|^{\frac12+\epsilon}}$. But we will focus on the analytic case in this paper.\newline

\noindent {\bf Notations.} We will write $\rmm_{\bT^d}$ for the Lebesgue probability measure on $\bT^d$, and $e(\theta)$ for the function $e^{2\pi i\theta}$. The norm $\|\theta\|$ of $\theta\in\bT^1$ is the distance from $\theta$ to $0$. This norm extends naturally to $\theta\in\bR$, that is, $\|\theta\|=\mathrm{dist}(\theta,\bZ)$. A basic fact that will be used repeatedly is $\|\theta\|\ll e(\theta)-1\ll\|\theta\|$.\newline

\noindent {\bf Acknowledgments.} I am deeply grateful to Peter Sarnak for encouragement and insightful discussions. Thanks are also due to Joanna Ku\l{}aga-Przymus and Mariusz Lema\'nczyk for valuable comments on an earlier draft. This work was carried out during a research stay as a member at the Institute for Advanced Study, and I would like to thank the IAS for its support. This work was also supported by NSF grants DMS-1451247 and DMS-1501295.

\section{Existence of a coboundary}\label{FurstenbergSec}

Remark that $T^n(x, y)=(x+n\alpha, y+H(n, x))$ where \begin{equation}\label{IterateEq}H(n, x)=\sum_{l=0}^{n-1}h(x+l\alpha).\end{equation}

One useful property is:
\begin{equation}\label{IterateEq2}H(n_1n_2, x)=\sum_{l=0}^{n_1-1}H(n_2, x+ln_2\alpha).\end{equation}

\begin{remark}\label{TrigBasis}By density of trigonometric polynomials in $C^0(\bT^2)$, it suffices to check for every $f(x,y)=e(\xi_1x+\xi_2y)$, $\xi_1, \xi_2\in\bZ$ in order to establish \eqref{MainEq}.\end{remark}

In this section, we make that the following reduction.

\begin{lemma}\label{DynaReduction} For a continuous map $h:\bT^1\mapsto\bT^1$, if \eqref{MainEq} fails, then $\alpha$ is irrational, and there exist:\begin{itemize}
\item $s\in\bN$;                                                                                                                     \item a pair of distinct primes $p_1$, $p_2$;
\item a measurable map $g:\bT^1\mapsto\bT^1$,                                                                                                                         \end{itemize}
such that
\begin{equation}\label{DynaReductionEq}s\big(H(p_1,x_0+p_1x)-H(p_2,x_0+p_2x)\big)=g(x+\alpha)-g(x), \text{for }\rmm_{\bT^1}\text{-a.e.} x.\end{equation}\end{lemma}

Here the map $H$ is defined by \eqref{IterateEq}.

Two tools are needed to prove this observation. The first is the Bourgain-Sarnak-Ziegler criterion for M\"obius disjointness, which is an extension of Vinogradov's bilinear method and Vaughan's identity, and eliminates the need of a saving of degree $\log^{-C}N$ in the hypothesis.

\begin{remark}As noted in \cite{BSZ13}, the Bourgain-Sarnak-Ziegler criterion is a dynamical interpretation of the bilinear method, which turns the study of type II sums into that of joinings between $T^{p_1}$ and $T^{p_2}$ where $p_1$, $p_2$ are different primes. In the case of skew products \eqref{DynaEq}, Ku\l{}aga-Przymus and Lema\'nczyk studied such joinings in \cite{KL13} as dynamical systems on $\bT^3$. A simple observation is that to study the correlation in Proposition \ref{BSZ} below, it suffices to study a certain $\bT^2$-factor of the dynamics on $\bT^3$, which would deduce Lemma \ref{DynaReduction} from \cite{KL13}*{Theorem 2.5.2}. We include the proof of Lemma \ref{DynaReduction} for completeness.
\end{remark}

\begin{proposition}\label{BSZ}\cite{BSZ13} Suppose two functions $F, \nu:\bN\mapsto\bC$ satisfy $|F|\leq 1$, $|\nu|\leq 1$ and $\nu$ is multiplicative. If for any pair of distinct primes $p_1$, $p_2$, $$\sum_{n\leq N}F(p_1n)\overline{F(p_2n)}=o(N),$$ then $$\sum_{n\leq N}\nu(n)F(n)=o(N).$$\end{proposition}

The second tool we need is the following dynamical dichotomy proved by Furstenberg.
\begin{proposition}\label{Furstenberg}\cite{F61} Suppose  $\{\Omega_0, T_0, \mu_0\}$ is uniquely ergodic topological dynamical system with $\mu_0$ being the unique ergodic measure, and $h:\Omega_0\mapsto\bT^1$ is a continuous function. Let $T:\Omega_0\times\bT^1\mapsto\Omega_0\times\bT^1$ be defined by $T(x,y)=(T_0(x), y+h(x))$. Then exactly one of the following is true:
\begin{enumerate}
\item $T$ is ergodic with $\mu_0\times\rmm_{\bT^1}$ being the unique ergodic measure.
\item There exists a measurable map $g:\Omega\mapsto\bT^1$ and a non-zero integer $s$ such that $sh(x)=g(T_0(x))-g(x)$ for $\mu$-almost every $x$.
\end{enumerate}
\end{proposition}

Notice that in any case, $\mu_0\times\rmm_{\bT^1}$ is $T$-invariant. And in case (2), for any $y_0\in\bT^1$, the closed set $\{(x,y): k(y-y_0)=f(x)\}$ is $T$ invariant. As this set is a measurable $k$-fold cover of $\Omega_0$, it naturally supports a $T$-invariant probability measure that projects to $\mu_0$ and assigns equal mass to $k$ distinct points in each fiber along $y$-direction.

\begin{proof}[Proof of Lemma \ref{DynaReduction}] Suppose that \eqref{MainEq} fails for some $f$. The $\alpha$ must be irrational (see e.g. \cite{LS15}). By Remark \ref{TrigBasis}, we can assume $f(x,y)=e(\xi_1x+\xi_2y)$. Furthermore, $\xi_2$ cannot be $0$, as otherwise $$\sum_{n\leq N}\mu(n)f(T^n(x_0,y_0))=\sum_{n\leq N}\mu(n)e(\xi_1(x_0+n\alpha))=e(\xi_1x_0)\sum_{n\leq N}\mu(n)e(n\xi_1\alpha),$$ which is  uniformly $o(N)$ for all $\alpha$ by Davenport's estimate \cite{D37}.

By Proposition \ref{BSZ} there are distinct primes $p_1$, $p_2$ such that
\begin{equation}\label{DynaReductionEq1}\frac1N\sum_{n\leq N}f(T^{p_1n})(x_0,y_0)\overline{f(T^{p_2n})(x_0,y_0)}\nrightarrow 0.\end{equation}
By \eqref{IterateEq}, the left hand side is equal to:
\begin{equation}\label{DynaReductionEq2}
\begin{aligned}
&\frac1N\sum_{n\leq N}e\big(\xi_1(x_0+p_1n\alpha)+\xi_2(y_0+H(p_1n,x_0))\big)\\
&\hskip2cm e\big(-\xi_1(x_0+p_2n\alpha)-\xi_2(y_0+H(p_2n,x_0))\big)\\
=&\frac1N\sum_{n\leq N}e\Big(\xi_1(p_1-p_2)n\alpha+\xi_2\big(H(p_1n,x_0)-H(p_2n,x_0)\big)\Big)\\
=&\frac1N\sum_{n\leq N}f\Big((p_1-p_2)n\alpha, \big(H(p_1n,x_0)-H(p_2n,x_0)\big)\Big).
\end{aligned}
\end{equation}

Define a new dynamical system of the form \eqref{DynaEq} by
$$\tilde T(x,y)=\big(x+\alpha, y+H(p_1,x_0+p_1x)-H(p_2,x_0+p_2x)\big).$$
We claim that \begin{equation}\label{DynaReductionEq4}\tilde T^n(0,0)=\big(n\alpha, H(p_1n,x_0)-H(p_2n,x_0)\big)\end{equation} for all $n\geq 0$. In fact, this is true for $n=0$. Assume for induction that \eqref{DynaReductionEq4} holds for $n$, then $\tilde T^{n+1}(0,0)$ equals
$$\Big(n\alpha+\alpha, H(p_1n,x_0)-H(p_2n,x_0)+H(p_1,x_0+p_1n\alpha)-H(p_1,x_0+p_1n\alpha)\Big).$$
Because by \eqref{IterateEq2}, for all $p\in\bN$,
$$\begin{aligned}
H(pn,x_0)+H(p,x_0+pn\alpha)
=&\sum_{l=0}^{n-1}H(p, x_0+ln\alpha)+H(p,x_0+pn\alpha)\\
=&\sum_{l=0}^{n}H(p, x_0+ln\alpha)=H\big(p(n+1),x_0\big),
\end{aligned}$$
\eqref{DynaReductionEq4} holds for $n+1$ as well.

It follows that, for $\tilde f(x,y)=f\big((p_1-p_2)x, y\big)$,
\begin{equation}\eqref{DynaReductionEq2}=\frac1N\sum_{n\leq N}\tilde f\big(\tilde T^n(0,0)\big),\end{equation}

This implies that $\rmm_{\bT^2}$ cannot be the unique ergodic probability measure for $\tilde T$. Suppose otherwise, then any weak$^*$ limit of a subsequence of $\{\frac1N\sum_{n\leq N}\delta_{\tilde T^n(0,0)}\}_{N=1}^\infty$, which is a $\tilde T$-invariant probability measure, must equal $\rmm_{\bT^2}$. Because the space of probability measures on $\bT^2$ is weak$^*$ compact, $\frac1N\sum_{n\leq N}\delta_{\tilde T^n(0,0)}\rightarrow\rmm_{\bT^2}$.  By Birkhoff's ergodic theorem, \eqref{DynaReductionEq2} converges to $$\begin{aligned}
\iint\tilde f(x,y)\mathrm dx\mathrm dy=&\iint f\big((p_1-p_2)x,y\big)\mathrm dx\mathrm dy=\iint f(x,y)\mathrm dx\mathrm dy\\
=&\int e(\xi_1x)\mathrm dx\cdot\int e(\xi_2y)\mathrm dy,
\end{aligned}
$$ which equals $0$ because $\xi_2\neq 0$. This contradicts \eqref{DynaReductionEq1}.

Since $\alpha$ is irrational, $x\mapsto x+\alpha$ is uniquely ergodic on $\bT^1$. By applying Proposition \ref{Furstenberg} to the map $\tilde T$, we know that there are $s\in\bN$ and a measurable map $g:\bT^1\mapsto\bT^1$ for which \eqref{DynaReductionEq} holds for $\rmm_{\bT^1}$-almost every $x$.
\end{proof}

Recall that the degree of a continuous map $h:\bT^1\mapsto\bT^1$ is the image $d=h_*1\in\pi_1(\bT^1)=\bZ$, where $1$ is the identity in $\pi_1(\bT^1)$.

\begin{corollary}\label{DegreeZero}If a Lipschitz continuous map $h:\bT^1\mapsto\bT^1$ makes \eqref{MainEq} fail, then $h$ is homotopically trivial.\end{corollary}

The corollary follows from \cite{KL13}*{Remark 2.5.7}, we again give the proof for completeness.

\begin{proof}Recall first a few simple facts:\begin{enumerate}
\item If $h$ has degree $d$, then so does $x\mapsto h(x+a)$ for all $a\in\bR$.
\item If $h$ has degree $d$, then $x\mapsto h(ax)$ has degree $ad$ for all $a\in\bZ$.
\item If $h_1$ and $h_2$ respectively have degrees $d_1$, $d_2$, then the degree of $h_1+h_2$ is $d_1+d_2$.
\end{enumerate}

It follows that if $h$ has degree $d$, then the map
$$\begin{aligned}&H(p_1,x_0+p_1x)-H(p_2,x_0+p_2x)\\
=&\sum_{l=0}^{p_1-1}h(x_0+p_1x+l\alpha)-\sum_{l=0}^{p_2-1}h(x_0+p_2x+l\alpha)\end{aligned}$$ has degree $(p_1^2-p_2^2)d$.

Assume, in order to get a contradiction, that $d\neq 0$, then for any two different primes $p_1$, $p_2$, $H(p_1,x_0+p_1x)-H(p_2,x_0+p_2x)$ is homotopically non-trivial. Since $h$ is assumed to be Lipschitz, by \cite{F61}*{Lemma 2.2}, there is no solution to \eqref{DynaReductionEq}. So, by Lemma \ref{DynaReduction}, \eqref{MainEq} must hold. This contradicts the hypothesis and concludes the proof.\end{proof}

\begin{remark}\label{FourierRmk} If $h$ is as in Corollary \ref{DegreeZero}, then it can be realized as a Lipschitz function from $\bR/\bZ$ to $\bR$. In particular, one can talk about its Fourier series, which pointwise converges to $h$ because of Lipschitz continuity, i.e., for all $x\in\bT^1$,
\begin{equation}\label{FourierEq}h(x)=\sum_{m\in\bZ}\hh(m)e(mx).\end{equation}\end{remark}

\section{Eliminating non-resonant frequencies}\label{DioReductionSec}

From now on, we suppose that \eqref{MainEq} is not true. In addition, assume that $h$ is analytic, and hence \eqref{FourierEq} holds and there exists $\tau >0$ for which \begin{equation}\label{AnaEq}|\hh(m)|\ll e^{-\tau |m|},\end{equation} where the implied constant may depend on $h$ and $\tau$.

Recall that $\alpha$ must be irrational in this case. Let $\frac{p_k}{q_k}$ be the convergents given by $\alpha$'s continued fraction expansion $[0; a_1, a_2, \cdots]$.

\begin{remark}\label{DioRmk}We list below a few fundamental properties of approximation of $\alpha$ by $p_k, q_k$ (see \cite{K97}):
\begin{enumerate}
\item $p_{k+1}=a_kp_k+p_{k-1}$,  $q_{k+1}=a_kq_k+q_{k-1}$ and $(p_k,q_k)=1;$
\item $\frac1{q_{k+1}+q_k}<\|q_k\alpha\|<\frac1{q_{k+1}};$
\item If for integers $m$ and $n$, $|\alpha-\frac nm|<\frac1{2m^2}$, then $\frac nm=\frac {p_j}{q_j}$ for some $j$.
\end{enumerate}
\end{remark}

For all $b_1, b_2\in\bN$, define \begin{equation}\label{FourierSupportEq}M_{b_1, b_2}=\displaystyle\bigcup_{\substack{k: q_{k+1}>e^{\frac\tau2 q_k}\\ q_k\geq b_1}}\{m: q_k\leq |m|<b_2q_{k+1}\text{ and }q_k|m\}\end{equation}

The following lemma is similar to \cite{LS15}*{Lemma 4.1}.

\begin{lemma}\label{DioFrequency}Suppose $h:\bT^1\mapsto \bR$ satisfies \eqref{AnaEq}. Then for all $b_1, b_2\in\bN$, the series $$\sum_{m\notin M_{b_1, b_2} \cup\{0\}}\hh(m)\frac{1}{e(m\alpha)-1}e(mx)$$ converges uniformly and hence defines a continuous function $\phi(x)$. \end{lemma}
\begin{proof} Since $M_{b_1, 1}\subset M_{b_1, b_2}$, it suffices to prove for $b_2=1$.  If $m\notin M_{b_1, 1}\cup\{0\}$, then at least one of the following three cases holds:

(1) $|m|<\min_{q_k\geq b_1}q_k $. Since only finitely many frequencies are involved in this case, they can be ignored without affecting the convergence.

(2) $q_k\leq |m|<q_{k+1}$, $q_k\nmid |m|$.  In this case, we claim that $\|m\alpha\|\geq\frac1{2|m|}$. By Remark \ref{DioRmk}, $\|m\alpha\|\geq\frac1{2m}$ unless: $|m|=lq_j$ and $\|m\alpha\|=|m\alpha-lp_j|$ for some $j\leq k$ and $l\in\bZ\backslash\{0\}$. In the later case, by the assumption on $m$, $j$ must be less than $k$, and we still have $$\|m\alpha\|=|l|\cdot |q_j\alpha-p_j|=|l|\cdot\|q_j\alpha\|>\frac {|l|}{q_{j+1}+q_j}\geq \frac1{q_k+q_{k-1}}>\frac 1{2q_k}\geq \frac1{2m}.$$

It follows that
\begin{equation}\label{DioFrequencyEq1}\begin{aligned}
&\sum_{\substack{q_k\leq |m|< q_{k+1}\\q_k\nmid m}}\left|\hh(m)\frac{1}{e(m\alpha)-1}e(mx)\right|\\
\ll &\sum_{\substack{q_k\leq |m|< q_{k+1}\\q_k\nmid m}} (e^{-\tau |m|}\cdot |m|\cdot 1)
\leq 2\sum_{m=q_k}^{\infty}e^{-\tau  m}m\\
\ll& e^{-\frac{\tau }2q_k}.
\end{aligned}\end{equation}

(3) $q_k\leq |m|<q_{k+1}\leq e^{\frac\tau2 q_k}$ and $q_k|m$. In this case, $|m|=aq_k$ with $1\leq a\leq a_k$. By Remark \ref{DioRmk}, $\frac a{q_{k+1}+q_k}<a\|q_k\alpha\|<\frac a{q_{k+1}}\leq\frac{a_k}{q_{k+1}}<\frac1{q_k}$, and thus $a\|q_k\alpha\|=\|m\alpha\|$. Therefore, because $q_{k+1}\leq e^{\frac\tau2 q_k}$

\begin{equation}\label{DioFrequencyEq2}\begin{aligned}
&\sum_{\substack{q_k\leq |m|< q_{k+1}\\q_k| m}}\left|\hh(m)\frac{1}{e(m\alpha)-1}e(mx)\right|\\
\ll &\sum_{\substack{|m|=aq_k\\1\leq a\leq a_k}} (e^{-\tau |m|}\cdot \frac{q_{k+1}+q_k}a\cdot 1)
\ll 2\sum_{a=1}^{\infty}e^{-\tau  a q_k}e^{\frac\tau2 q_k}\\
\ll& e^{-\tau q_k}e^{\frac\tau2 q_k}=e^{-\frac\tau2q_k}.
\end{aligned}\end{equation}

Since both \eqref{DioFrequencyEq1} and \eqref{DioFrequencyEq2} are convergent when summed over all the $q_k$'s. The lemma follows.\end{proof}

In the next reduction, we eliminate the non-resonant frequencies, i.e., those not in $M_{b_1, b_2}$.

\begin{corollary}\label{FourierReduction} Suppose $h$ is as in Lemma \ref{DioFrequency}.  For arbitrary $b_1, b_2\in\bN$ let $h_1(x)=\sum_{m\in M_{b_1, b_2}\cup\{0\}}\hh(m)e(mx)$. Then: \begin{enumerate}
\item $T$ satisfies \eqref{MainEq} if and only $T_1(x,y)=(x+\alpha, y+h_1(x))$ does;
\item For distinct primes $p_1, p_2$, $h$ satisfies \eqref{DynaReductionEq}  for a measurable map $g:\bT^1\mapsto\bT^1$ if and only if $h_1$ satisfies the same condition with respect to another measurable map $g_2$.
\end{enumerate} \end{corollary}
\begin{proof} Write $h(x)=h_1(x)+h_2(x)$, where $h_1$ and $h_2$ are partial sums from the Fourier expansion of $h$, respectively over $M_{b_1, b_2}\cup\{0\}$ and its complement.

Let $\phi(x)$ be given by Lemma \ref{DioFrequency}. Then
$$\begin{aligned}\phi(x+\alpha)-\phi(x)=&\sum_{m\notin M_{b_1, b_2}\cup\{0\}}\hh(m)\frac1{e(m\alpha)-1}\big(e(mx+m\alpha)-e(mx)\big)\\
=&\sum_{m\notin M_{b_1, b_2}\cup\{0\}}\hh(m)\frac1{e(m\alpha)-1}\big(e(m\alpha)-1\big)e(mx)\\
=& h_2(x).
\end{aligned}$$

Denote $\Phi(x,y)=(x, y-\phi(x))$. Then $\Phi^{-1}(x,y)=(x, y+\phi(x))$, and
$$\begin{aligned}(\Phi^{-1}\circ T_1\circ\Phi)(x,y)=&(\Phi^{-1}\circ T_1)\big(x,y-\phi(x)\big)\\
=&\Phi^{-1}\big(x+\alpha,y-\phi(x)+h_1(x)\big)\\
=&\big(x+\alpha,y-\phi(x)+h_1(x)+\phi(x+\alpha)\big)\\
=&\big(x+\alpha,y+h_1(x)+h_2(x)\big)=\big(x+\alpha,y+(x)\big)\\
=&T(x,y).
\end{aligned}$$
Therefore, $f(T^n(x_0,y_0))=f_1(T_1^n(x_1,y_1))$, where $f_1=f\circ\Phi^{-1}$ and $(x_1,y_1)=\Phi(x_0,y_0)$. This shows claim (1)

As for claim (2), define $H_1(p,x)$, $H_2(p,x)$ as in \eqref{IterateEq} respectively using $h_1$ and $h_2$. Then for any $p$, $H(p, x_0+px)=H_1(p, x_0+px)+H_2(p,x_0+px)$. However, $H_2(p,x)=\sum_{l=0}^{p-1}\big(\phi(x+(l+1)\alpha)-\phi(x+l\alpha)\big)=\phi(x+p\alpha)-\phi(x)$.

Therefore,
$$\begin{aligned}
&s\big(H(p_1,x_0+p_1x)-H(p_2,x_0+p_2x)\big)\\
&\hskip2cm -s\big(H_1(p_1,x_0+p_1x)-H_1(p_2,x_0+p_2x)\big)\\
=&s\big(H_2(p_1,x_0+p_1x)-H_2(p_2,x_0+p_2x)\\
=&\phi(x_0+px+p\alpha)-\phi(x_0+px)=g_2(x+\alpha)-g_2(x).
\end{aligned}$$
where $g_2(x)=\phi(x_0+px)$. So \eqref{DynaReductionEq} holds for $H$ and $g$ if and only if it holds for $H_1$ and $g_1$ with $g_1=g-g_2$.
\end{proof}

\section{Almost linearity of orbit along arithmetic progressions}\label{AlmostLinearSec}

The aim of this section is to show the orbit point $T^n(x_0,y_0)\in\bT^2$ varies in an almost linear way when $n$ assumes value in an arithmetic progression of step length $q_k$.

\begin{lemma}\label{CocycleEstimate}Suppose an analytic function $h:\bT^1\mapsto\bR$ satisfies \eqref{AnaEq}, and there are $b_1\geq b_2\geq 1$ such that $\hh$ is supported on $M_{b_1, b_2}\cup\{0\}$. Then the estimate
$$\left|H(q_k,x)-q_k\hh(0)\right|\ll e^{-\frac\tau4 q_k}$$ holds uniformly for all $x\in\bT^1$ and for all $k$ satisfying $q_{k+1}>e^{\frac\tau2 q_k}$, where the implied constant depends on $h$, $\tau$ and $b_1$, $b_2$.\end{lemma}

\begin{proof}Using geometric series, one easily finds out that for all $n\in\bN$,  \begin{equation}\label{IterateFourierEq}\begin{aligned}H(n, x)=&\sum_{l=0}^{n-1}h(x+l\alpha)=\sum_{l=0}^{n-1}\sum_{m\in\bZ}\hh(m)e(mx+lm\alpha)\\
=&n\hh(0)+\sum_{m\neq 0}\hh(m)\frac{e(nm\alpha)-1}{e(m\alpha)-1}e(mx).\end{aligned}\end{equation}
Because $\hh$ is supported on $M_{b_1,b_2}\cup\{0\}$,
\begin{equation}\label{CocycleEstimateEq1}\begin{aligned}\left|H(q_k, x)-q_k\hh(0)\right|
=&\left|\sum_{m\in M_{b_1,b_2}}\hh(m)\frac{e(q_km\alpha)-1}{e(m\alpha)-1}e(mx)\right|\\
\leq &\sum_{\substack{j: q_{j+1}>e^{\frac\tau2 q_j}\\ q_j\geq b_1}}\sum_{\substack{q_j\leq m<b_2q_{j+1}\\q_j|m}}\left|\hh(m)\frac{e(q_km\alpha)-1}{e(m\alpha)-1}e(mx)\right|\\
=&\sum_{\substack{j: q_{j+1}>e^{\frac\tau2 q_j}\\ q_j\geq b_1}}\sum_{a=1}^{b_2a_j}\left|\hh(aq_j)\frac{e(aq_jq_k\alpha)-1}{e(aq_j\alpha)-1}e(aq_jx)\right|\\
=:&\sum_{\substack{j: q_{j+1}>e^{\frac\tau2 q_j}\\ q_j\geq b_1}}|H_{j,k}(x)|.
\end{aligned}
\end{equation}

We estimate $|H_{j,k}(x)|$ separately according to $j<k$, or $j\geq k$.

(1) $j<k$. Remark $a\|q_j\alpha\|\leq a q_{j+1}^{-1}\leq b_2a_jq_{j+1}^{-1}<b_2q_j^{-1}\leq b_2b_1^{-1}<1$. This implies $\|aq_j\alpha\|=a\|q_j\alpha\|\geq a(q_{j+1}+q_j)^{-1}$.

On the other hand, $\|aq_jq_k\alpha\|\leq aq_j\|q_k\alpha\|\leq aq_jq_{k+1}^{-1}$. By these estimates and the analyticity assumption \eqref{AnaEq},
$$|H_{j,k}(x)|\ll\sum_{a=1}^{b_2a_j}e^{-\tau aq_j}\cdot\frac{aq_jq_{k+1}^{-1}}{a(q_{j+1}+q_j)^{-1}}
\ll\sum_{a=1}^{\infty}e^{-\tau aq_j}q_jq_{j+1}q_{k+1}^{-1}
\ll  q_{j+1}q_{k+1}^{-1}.$$
Thus by the assumption $q_{k+1}>e^{\frac\tau2 q_k}$,
\begin{equation}\label{CocycleEstimateEq2}\sum_{\substack{j<k: q_{j+1}>e^{\frac\tau2 q_j}\\  q_j\geq b_1}}|H_{j,k}(x)|\ll\sum_{j=1}^{k-1}q_{j+1}q_{k+1}^{-1}\ll q_kq_{k+1}^{-1}<q_ke^{-\frac\tau2q_k}\ll e^{-\frac\tau4q_k}.\end{equation}
Where we used that $\sum_{j=1}^Jq_j\leq 4q_J$, which is due to the basic fact that $q_{j+2}\geq q_{j+1}+q_j>2q_j$ for all $j$.

(2) $j\geq k$. In this case, we use the trivial estimate that $\left|\frac{e(aq_jq_k\alpha)-1}{e(aq_j\alpha)-1}\right|\leq q_k$. This implies:
$$|H_{j,k}(x)|\ll\sum_{a=1}^{b_2a_j}e^{-\frac\tau2 aq_j}\cdot q_k\ll e^{-\frac\tau2 q_j}q_k,$$
and thus,
\begin{equation}\label{CocycleEstimateEq3}\sum_{\substack{j\geq k: q_{j+1}>e^{\frac\tau2 q_j}\\ q_j\geq b_1}}|H_{j,k}(x)|\ll\sum_{q=q_k}^\infty e^{-\frac\tau2 q_k}q_k\ll e^{-\frac\tau2 q_k}q_k\ll e^{-\frac\tau4q_k}.\end{equation}

The lemma is established by combining \eqref{CocycleEstimateEq1}, \eqref{CocycleEstimateEq2} and \eqref{CocycleEstimateEq3}.
\end{proof}

\section{Almost periodicity of orbit}\label{AlmostPeriodicSec}

We know from the previous lemma that when $n$ moves along arithmetic progressions of step length $q_k$, the $y$-coordinate of $T^n(x, y)$ is close to a linear sequence in $\bT^1$ of slope $\hh(0)$, which is the average value of the function $h$. Our next goal is to control this average. Unlike the degree $d$, $\hh(0)$ cannot be fully eliminated. However, it exhibits similar Diophantine patterns to those of $\alpha$.

\begin{lemma}\label{RotationNumber} Suppose $h$ is as in Lemma \ref{CocycleEstimate}, and \eqref{MainEq} is not true, then there exists $\eta\in (0,\frac\tau 4)$, $S\in\bN$, such that for any given $\delta>0$:

For all sufficiently large $k$ with $q_{k+1}>e^{\frac\tau2q_k}$ and any positive integer $a\leq e^{\eta q_k}$, $$\big\|aSq_k\hh(0)\big\|<\delta.$$\end{lemma}

\begin{proof} By Lemma \ref{DynaReduction}, there exists $s\in\bN$ and distinct primes $p_1<p_2$, and a measurable map $g:\bT^1\mapsto\bT^1$ such that \eqref{DynaReductionEq} holds.

By Lemma \ref{FourierReduction}, after replacing $h$ by $\sum_{m\in M_{p_2, 1}\cup\{0\}}(x)$,  \eqref{DynaReductionEq} still holds for the same $p_1$, $p_2$ and a different $g$. In addition, $\hh(0)$ does not change by this modification. So we can assume $\hh$ is supported on  $M_{p_2, 1}\cup\{0\}$.

By Luzin's theorem, there is a compact set $\Omega\subset\bT^1$ of Lebesgue measure at least $0.9$ on which $g$ is continuous. There is $\epsilon$ such that if $\|x-x'\|<\epsilon$ and $x, x'\in\Omega$ then $\|g(x')-g(x)\|<\frac\delta2$. Without loss of generality, we will assume $\epsilon<\frac\delta2$.

For $\eta<\frac\tau 4$, under the assumptions on $k$ and $a$,  $a<e^{-\frac\tau4 q_k} q_{k+1}$,  and thus, $$\|aq_k\alpha\|<a\|q_k\alpha\|<e^{-\frac\tau4 q_k} q_{k+1}\cdot q_{k+1}^{-1}=e^{-\frac\tau4 q_k},$$ and is less than $\epsilon$ when $k$ is sufficiently large.
Consider the pair of points $x$ and $x+aq_k\alpha$, they both belong to $\Omega$ if $x$ lies in $\Omega'=\Omega\cap(\Omega-aq_k\alpha)$. As $\rmm_{\bT^1}(\Omega-aq_k\alpha)=\rmm_{\bT^1}(\Omega)\geq 0.9$, we know $\Omega'$ is non-empty.

Write $\psi(x)$ for $g(x+\alpha)-g(x)$, which by \eqref{DynaReductionEq} equals  $s\big(H(p_1,x_0+p_1x)-H(p_2,x_0+p_2x)\big)$ for almost every $x\in\Omega'$.

Fix hereafter such a generic point $x\in\Omega'$, then \begin{equation}\label{RotationNumberEq1}\big\|g(x+aq_k\alpha)-g(x)\big\|<\frac\delta2.\end{equation}

We apply \eqref{IterateFourierEq} to get the Fourier expansion of $\psi$:
$$\psi(x)=s\sum_m\hh(m)\Big(\sum_{l=0}^{p_1}e(lm\alpha+mx_0+p_1mx) -\sum_{l=0}^{p_2}e(lm\alpha+mx_0+p_2mx)\Big).$$
In other words,
$$\hat\psi(m)=s\left(\hh\Big(\frac m{p_1}\Big)\sum_{l=0}^{p_1}e(lm\alpha+mx_0)-\hh\Big(\frac m{p_2}\Big)\sum_{l=0}^{p_2}e(lm\alpha+mx_0)\right),$$
with the convention that $\hh(\theta)=0$ if $\theta\notin\bZ$.

This implies several properties of $\psi$. First, as $h$ satisfies \eqref{AnaEq},
\begin{equation}\label{RotationNumberEq2}|\hat\psi(m)|\leq s\left(p_1\Big|\hh\Big(\frac m{p_1}\Big)\Big|-p_1\Big|\hh\Big(\frac m{p_1}\Big)\Big|\right)\ll e^{-\frac\tau{p_2}|m|},\end{equation} with an implied constant depending on $h$, $s$ and $p_2$.

Second,  \begin{equation}\label{RotationNumberEq3}\supp\hat\psi\subset p_1M_{p_2,1}\cup p_2M_{p_2,1}\subset M_{p_2, p_1}\cup M_{p_2, p_2}=M_{p_2, p_2}.\end{equation}

Furthermore,
\begin{equation}\label{RotationNumberEq4}\hat\psi(0)=s(p_1-p_2)\hh(0).\end{equation}

Construct $\Psi(n,x)$ as in \eqref{IterateEq} with $h$ replaced by $\psi$, i.e. $\Psi(n,x)=\sum_{l=1}^{n-1}\psi(x+l\alpha)$.

Thanks to \eqref{RotationNumberEq2} and \eqref{RotationNumberEq3}, we can apply Lemma \ref{CocycleEstimate} to $\Psi$ to assert that, uniformly for all $x\in\bT^1$,
\begin{equation}\label{RotationNumberEq5}\big|\Psi(q_k, x)-q_k\hat\psi(0)\big|\ll e^{-\frac\tau{4p_2}q_k}.\end{equation}
On the other hand, notice that
\begin{equation}\label{RotationNumberEq6}g(x+aq_k\alpha)-g(x)=\Psi(aq_k,x)
=\sum_{l=0}^{a-1}\Psi(q_k,x+lq_k\alpha).
\end{equation}
From \eqref{RotationNumberEq1}, \eqref{RotationNumberEq5} and \eqref{RotationNumberEq6}, we deduce
$$\begin{aligned}
&\Big\|aq_k\hat\psi(0)\Big\|\\
\leq &\Big\|\sum_{l=0}^{a-1}\Psi(q_k,x+lq_k\alpha)-aq_k\hat\psi(0)\Big\|+\Big\|\sum_{l=0}^{a-1}\Psi(q_k,x+lq_k\alpha)\Big\|\\
\leq&\sum_{l=0}^{a-1}\big|\Psi(q_k,x+lq_k\alpha)-aq_k\hat\psi(0)\big|+\big\|g(x+aq_k\alpha)-g(x)\big\|\\
<&Cae^{-\frac\tau{4p_2}q_k}+\frac\delta2,
\end{aligned}$$
where $C$ is a constant that depend only on $h$, $s$ and $p_2$.

Choose $\eta<\frac\tau{8p_2}$. Then
$\Big\|aq_k\hat\psi(0)\Big\|<C\cdot e^{\frac\tau{8p_2}q_k}\cdot e^{-\frac\tau{4p_2}q_k}+\frac\delta2$ is less than $\delta$ for sufficiently large $k$.

However, by \eqref{RotationNumberEq4}, $\big\|aq_k\hat\psi(0)\big\|=\|aq_ks(p_1-p_2)\hh(0)\|$. So the lemma is true with $S=s(p_1-p_2)$.\end{proof}

This, together with Lemma \ref{CocycleEstimate}, shows that an orbit of $T$ is, within intervals of length up to $Sq_ke^{\eta q_k}$, almost periodic with period $Sq_k$.
\begin{corollary}\label{AlmostPeriodic}Let $\eta$ and $S$ be as in Lemma \ref{RotationNumber}. Then given any $\delta>0$, for all sufficiently large $k$ with $q_{k+1}>e^{\frac\tau2q_k}$ and any positive integer $a\leq e^{\eta q_k}$, $$d\big(T^{aSq_k}(x,y), (x,y)\big)<\delta,\ \forall(x,y)\in\bT^2.$$
\end{corollary}
\begin{proof}$T^{aSq_k}(x,y)-(x,y)=\big(aSq_k\alpha, H(aSq_k,x)\big).$ So it suffices to bound both $\|aSq_k\alpha\|$ and $\|H(aSq_k,x)\|$ by $\frac\delta 2$.

First, $\|aSq_k\alpha\|\leq aS\|q_k\alpha\|<e^{\frac\tau4 q_k}Sq_{k+1}^{-1}<e^{\frac\tau4 q_k}Se^{-\frac\tau2 q_k}=Se^{-\frac\tau4 q_k}.$ As $S$ is independent of $k$, this is less than $\frac\delta2$ when $q_k$ is large enough.

On the other hand, by Lemma \ref{CocycleEstimate},
$$\label{AlmostPeriodicEq1}\begin{aligned}
&\|H(aSq_k,x)\|\\
=&\Big\|\sum_{l=1}^{aS-1}H(q_k, x+lq_k)\Big\|
=\Big\|\sum_{l=1}^{aS-1}\big(H(q_k, x+lq_k)-q_k\hh(0)\big)+aSq_k\hh(0)\Big\|\\
\leq&\sum_{l=1}^{aS-1}\big|H(q_k, x+lq_k)-q_k\hh(0)\big|+\big\|aSq_k\hh(0)\big\|\\
\leq&CaSe^{-\frac\tau4q_k}+\big\|aSq_k\hh(0)\big\|,
\end{aligned}$$
where $C$ is a constant independent of $k$.
By Lemma \ref{RotationNumber}, $\|aSq_k\hh(0)\|<\frac\delta4$ when $k$ is sufficiently large. As for the first term, as $a<e^{-\eta q_k}$, $\eta<\frac\tau 4$ and $C$, $S$ do not depend on $k$,
$CaSe^{-\frac\tau4q_k}\ll e^{-(\frac\tau4-\eta)q_k}$ and also becomes less than $\frac\delta 4$ for large $k$. Therefore $\|H(aSq_k,x)\|<\frac\delta 2$. This completes the proof.
\end{proof}

\section{Averages in short intervals}\label{ShortAvg}
We continue to assume that \eqref{AnaEq} is true but \eqref{MainEq} is not, and try to obtain a contradiction.

Disjointness from $\mu(n)$ is known for periodic sequences. And in light of the previous section, we know $T^n(x_0, y_0)$ is not far from a periodic sequence. In this part we show that the error obtained there does not accumulate to destroy the disjointness.

To begin with, remark that we can always assume there are infinitely many $k$'s with $q_{k+1}>e^{\frac\tau 2q_k}$. Otherwise, Corollary \eqref{FourierReduction} allows to assume $\hh$ is supported at $\{0\}$. Then $T(x, y)=(x+\alpha, y+\hh(0))$ is a rotation of $\bT^2$, for which \eqref{MainEq} is known to hold by Davenport's theorem \cite{D37}.

From now on, we fix a function $f(x,y)=e(\xi_1x+\xi_2y)$, which is allowed by Remark \ref{TrigBasis}. It suffices to show for all $\delta>0$ that, for sufficiently large $N$,
\begin{equation}\label{DeltaEq}\bE_{n< N}f(T^n(x_0,y_0))\mu(n)\ll \delta,\end{equation}
with an implied constant that is allowed to depend on $h$ and $f$.

Let $\eta$ and $S$ are given by Lemma \ref{AlmostPeriodic}.  Suppose $k$ is sufficiently large and satisfies $q_{k+1}>e^{\frac\tau4q_k}$. And set $A=\lfloor e^{\eta q_k} \rfloor$. Furthermore, let $N_0<\frac N2$ be an integer.

The average \eqref{MainEq} decomposes into averages in shorter intervals of length $ASq_k$:
$$\begin{aligned}&\bE_{n< N}\mu(n)f(T^n(x_0,y_0))\\
=&\bE_{L=N_0}^{N-1} \mu(n)f(T^n(x_0,y_0))+O(\frac{N_0}N)\\
=&\bE_{L=N_0}^{N-1}\bE_{n=L}^{L+ASq_k-1}\mu(n)f(T^n(x_0,y_0))+O(\frac {ASq_k}N)+O(\frac{N_0}N)\\
=&\bE_{L=N_0}^{N-1}\bE_{l=L}^{L+Sq_k-1}\bE_{a=0}^{A-1}\mu(l+aSq_k)f(T^{l+aSq_k}(x_0,y_0))+O(\frac {ASq_k}N+\frac{N_0}N).
\end{aligned}$$

As $a\leq A<e^{\eta q_k}$, by Lemma \ref{AlmostPeriodic}, for sufficiently large $k$, $$\big|f(T^{l+aSq_k}(x_0,y_0))-f(T^l(x_0,y_0))\big|<\|f\|_{C^1}\delta.$$ Therefore,

$$\begin{aligned}&\bE_{n< N}\mu(n)f(T^n(x_0,y_0))\\
=&\bE_{L=N_0}^{N-1}\bE_{l=L}^{L+Sq_k-1}\bE_{a=0}^{A-1}\mu(l+aSq_k)f(T^l(x_0,y_0))+O(\delta)+O(\frac {ASq_k}N+\frac{N_0}N).
\end{aligned}$$

For each $L$, construct a function $F_L:\bN\mapsto\bC$ by $F_L(n)=f(T^l(x_0,y_0))$ where $l$ is the unique integer in $[L, L+Sq_k)$ such that $l\equiv n(\mathrm{mod}\ Sq_k)$. Then $F_L$ is periodic with period $Sq_k$ and $|F_L|=1$, and
\begin{equation}\label{AvgDecomposeEq2}\begin{aligned}&\bE_{n< N}\mu(n)f(T^n(x_0,y_0))\\
=&\bE_{L=N_0}^{N-1}\bE_{n=L}^{L+ASq_k-1}\mu(n)F_L(n)+O(\delta+\frac {ASq_k}N+\frac{N_0}N).
\end{aligned}\end{equation}

We next decompose $\bE_{n=L}^{L+ASq_k-1}\mu(n)F_L(n)$ into short averages of Dirichlet characters.

\begin{lemma}\label{Dirichlet}For all $L, Q, A\in\bN$ and any periodic function $F:\bN\mapsto\bC$ of period $Q$ with $|F|\leq 1$,
$$\left|\bE_{L\leq n< L+AQ}\mu(n)F(n)\right|^2
\leq Q\bE_{\substack{d|Q\\ \chi\ \mathrm{mod}^*\ \frac Qd}}\Big|\bE_{\frac Ld\leq r<\frac Ld+A\frac Qd} \mu(r)\chi(r)\Big|^2,$$ where the first average on the right hand side is taken over all pairs $(d,\chi)$ such that $d|Q$ and $\chi$ is a Dirichlet character of conductor $\frac Qd$. \end{lemma}
\begin{proof} We group $n$ according to $(n,Q)$:
\begin{equation}\label{DirichletEq1}\begin{aligned}&\sum_{L\leq n< L+AQ}\mu(n)F(n)\\
=&\sum_{d|Q}\sum_{\substack{L\leq n< L+AQ\\(n,Q)=d}} \mu(n)F(n)
=\sum_{d|Q}\sum_{\substack{\frac Ld\leq r<\frac Ld+A\frac Qd\\(r,\frac Qd)=1}} \mu(rd)F(rd)\\
=&\sum_{d|Q}\sum_{\substack{\frac Ld\leq r<\frac Ld+A\frac Qd\\(r,\frac Qd)=(r,d)=1}} \mu(r
d)F(rd)
=\sum_{d|Q}\mu(d)\sum_{\substack{\frac Ld\leq r<\frac Ld+A\frac Qd\\(r,\frac Qd)=(r,d)=1}}  \mu(r)F(rd).
\end{aligned}\end{equation}
Here the requirement that $(r,d)=1$ is because otherwise $n=rd$ is not square-free and $\mu(rd)=0$.

Identify $\{r: \frac Ld\leq r<\frac Ld+\frac Qd,\ (r,\frac Qd)=1\}$ with the finite abelian group $\big(\bZ/(\frac Qd)\bZ\big)^\times$. The Dirichlet characters of conductor $\frac Qd$ form an orthonormal basis of the $l^2$-space on this group with respect to the uniform probability measure. Hence $F(rd)\mathbf 1_{(r,d)=1}$ can be decomposed as $\displaystyle\sum_{\chi\ \mathrm{mod}^*\  \frac Qd}w_{F,\chi}\chi$ on this group. Then, \begin{equation}\label{DirichletEq2}\sum_{\chi\ \mathrm{mod}^*\  \frac Qd}|w_{F,\chi}|^2=\|F(rd)\mathbf 1_{(r,d)=1}\|_{l^2}\leq \|F(rd)\|_{l^\infty}\leq 1.\end{equation}

It follows from this and \eqref{DirichletEq1} that,
$$\begin{aligned}
\left|\bE_{L\leq n< L+AQ}\mu(n)F(n)\right|^2
=&\frac 1{(AQ)^2}\left|\sum_{L\leq n< L+AQ}\mu(n)F(n)\right|^2\\
\leq&\frac 1{(AQ)^2}\left|\sum_{d|Q}\sum_{\chi\ \mathrm{mod}^*\  \frac Qd}w_{F,\chi}\sum_{\frac Ld\leq r<\frac Ld+A\frac Qd} \mu(r)\chi(r)\right|^2\\
= &\frac1{(AQ)^2}\left|\sum_{d|Q}\sum_{\chi\ \mathrm{mod}^*\ \frac Qd}\frac{AQ}d\omega_{F,\chi}\bE_{\frac Ld\leq r<\frac Ld+A\frac Qd} \mu(r)\chi(r)\right|^2\\
=&\left|\sum_{d|Q}\sum_{\chi\ \mathrm{mod}^*\ \frac Qd}\frac{\omega_{F,\chi}}d\bE_{\frac Ld\leq r<\frac Ld+A\frac Qd} \mu(r)\chi(r)\right|^2.\end{aligned}$$

We apply Cauchy-Schwarz inequality and get

$$\begin{aligned}
&\left|\bE_{L\leq n< L+AQ}\mu(n)F(n)\right|^2\\
\leq &\left(\sum_{d|Q}\sum_{\chi\ \mathrm{mod}^*\ \frac Qd}\frac{|\omega_{F,\chi}|^2}{d^2}\right)\left(\sum_{d|Q}\sum_{\chi\ \mathrm{mod}^*\ \frac Qd}\Big|\bE_{\frac Ld\leq r<\frac Ld+A\frac Qd} \mu(r)\chi(r)\Big|^2\right)\\
\stackrel{\eqref{DirichletEq2}}\leq &\left(\sum_{d|Q}\frac1{d^2}\right)\left(\sum_{d|Q}\sum_{\chi\ \mathrm{mod}^*\ \frac Qd}\Big|\bE_{\frac Ld\leq r<\frac Ld+A\frac Qd} \mu(r)\chi(r)\Big|^2\right)\\
\ll & Q\bE_{\substack{d|Q\\ \chi\ \mathrm{mod}^*\ \frac Qd}}\Big|\bE_{\frac Ld\leq r<\frac Ld+A\frac Qd} \mu(r)\chi(r)\Big|^2
\end{aligned},$$
where the last inequality is because $\sum_{d=1}^\infty\frac1{d^2}$ converges and there are $Q$ dirichlet characters $\chi$ of modulus $Q$.
\end{proof}

In order to control $\left|\bE_{\frac Ld\leq r<\frac Ld+A\frac Qd} \mu(r)\chi(r)\right|$, we make use of a recent theorem from \cite{MRT15} that allows to bound averages of a non-pretentious multiplicative function in random short intervals. For this purpose, we recall the definition of pretentiousness first.

For multiplicative functions $\nu, \nu':\bN\mapsto\bC$ whose absolute values are bounded by $1$, define
$$\bD(\nu, \nu', X)=\left(\sum_{p\leq X}\frac{1-\re(\nu(p)\overline{\nu'(p)})}p\right)^{\frac12}.$$
The function below measures how closely $\nu$ pretends to be $n^{it}$:
$$M(\nu, X)=\inf_{|t|\leq X}\bD(\nu,  n^{it}, X)^2.$$

One direction of Hal\'asz's Theorem \cite{H68} asserts that if $\sum_{n\leq X}\nu(n)=o(X)$, then for all $t\in\bR$,  $$\lim_{X\to\infty}D(\nu,  n^{it}, X)=\infty.$$ It is well known that (see e.g. \cite{IK04}*{(5.80)}) for a Dirichlet character of conductor $q$ and any $C>0$,  \begin{equation}\label{MuChiEq}\sum_{n\leq X}\mu(n)\chi(n)\ll_C q^{\frac12}(\log X)^{-C},\end{equation} we see that $\lim_{X\to\infty}D(\mu\chi, 1, X)=\infty$ for any given $\chi$.

It follows from \cite{MRT15}*{Lemma C.1}, which is based on an argument of Granville and Soundararajan, that
$$\inf_{|t|\leq X}\bD(\mu, \bar\chi n^{it}, X)\geq\frac14\min\big(\sqrt {\log\log X}, D(\mu, \bar\chi, X)\big) +O(1),$$ or equivalently,
$$\inf_{|t|\leq X}\bD(\mu\chi, n^{it}, X)\geq\frac14\min\big(\sqrt {\log\log X}, D(\mu\chi, 1, X)\big) +O(1).$$
It follows that:
\begin{equation}\label{PretentiousEq}\lim_{X\to\infty}M(\mu\chi,X)=\infty.\end{equation}

The following proposition, due to Matom\"aki, Radziwi\l\l{} and Tao, is the key input in our method and will be applied to $\nu(n)=\mu(n)\chi(n)$.

\begin{proposition}\label{MRT}\cite{MRT15}*{Theorem A.1}
Let $\nu$ be a multiplicative function with $|\nu|\leq 1$ and $X\geq l\geq 10$. Then
$$\bE_{X\leq L<2X}\left|\bE_{L\leq n<L+l}\nu(n)\right|^2\ll e^{-M(\nu,X)}M(\nu,X)+(\log X)^{-\frac1{50}}+\Big(\frac{\log\log l}{\log l}\Big)^2.$$\end{proposition}

\begin{proof}[Proof of Theorem \ref{Main}]

Let $J=\lceil\log_2\frac1\delta\rceil$. We choose $N_0=\lfloor 2^{-J}N\rfloor$, then $\frac{N_0}N<\delta$. Cut $[N_0, N)$ into $J$ dyadic intervals: $[2^{-j}N, 2^{-j+1}N)$ for $j=1,\cdots, J$.

For each pair $(d,\chi)$ and every $j$, by Proposition \ref{MRT},

\begin{equation}\label{MainPfEq1}
\begin{aligned}
&\bE_{2^{-j}N\leq L<2^{-j+1}N}\left|\bE_{\frac Ld\leq r<\frac Ld+\frac{ASq_k}d} \mu(r)\chi(r)\right|^2\\
\ll& \rho\left(\frac{2^{-j+1}N}d\right)+\left(\frac{\log\log \frac{ASq_k}d}{\log \frac{ASq_k}d}\right)^2,
\end{aligned}
\end{equation}
where $\rho(X)=\rho_\chi(X):=e^{-M(\mu\chi,X)}M(\mu\chi, X)+(\log X)^{-\frac1{50}}$ is a positive function that converges to $0$ as $X\to\infty$, thanks to \eqref{PretentiousEq}. However, we may make $\rho$ independent of $\chi$ (but dependent of $S$ and $q_k$), by taking $\rho(X)=\rho_{Sq_k}(X):=\max_{\chi\ \mathrm{mod}\ Sq_k}\rho_\chi(X)$. Furthermore, by replacing $\rho_{Sq_k}(X)$ with $\sup_{X'>X}\rho_{Sq_k}(X')$ we may also assume $\rho_{Sq_k}$ is decreasing.

Recall that $S$ and $\eta$ are fixed constants determined by the dynamical system $T$. Notice $d\leq Sq_k$, $2^{-j+1}\geq\delta$, and $A\leq e^{\eta q_k}$. Hence,
$$\eqref{MainPfEq1}\ll \rho_{Sq_k}\Big(\frac{\delta N}{Sq_k}\Big)+\left(\frac{\log\log A}{\log A}\right)^2\ll \rho_{Sq_k}\Big(\frac{\delta N}{Sq_k}\Big)+q_k^{-2}\log^2q_k.$$

After combining the dyadic intervals, we have
\begin{equation}\label{MainPfEq2}
\begin{aligned}&\bE_{L=N_0}^{N-1}\left|\bE_{\frac Ld\leq r<\frac Ld+\frac{ASq_k}d} \mu(r)\chi(r)\right|^2\ll \rho_{Sq_k}\Big(\frac{\delta N}{Sq_k}\Big)+q_k^{-2}\log^2q_k.\end{aligned}\end{equation}

It now follows that

\begin{equation}\label{MainPfEq4}\begin{aligned}&\bE_{n< N}\mu(n)f(T^n(x_0,y_0))\\
\stackrel{\eqref{AvgDecomposeEq2}}=&\bE_{L=N_0}^{N-1}\bE_{n=L}^{L+ASq_k-1}\mu(n)F_L(n)+O\Big(\delta+\frac {ASq_k}N+\frac{N_0}N\Big)\\
\stackrel{\text{(Cauchy-Schwarz)}}\ll &\left(\bE_{L=N_0}^{N-1}\left|\bE_{n=L}^{L+ASq_k-1}\mu(n)F_L(n)\right|^2\right)^{\frac12}+\delta+\frac {ASq_k}N\\
\stackrel{(\text{Lemma }\ref{Dirichlet})}\ll &\left(Sq_k\bE_{L=N_0}^{N-1}\bE_{\substack{d|Sq_k\\\chi\ \mathrm{mod}^*\ \frac{Sq_k}d}}\left|\bE_{n=L}^{L+ASq_k-1}\mu(n)\chi(n)\right|^2\right)^{\frac12}+\delta+\frac {ASq_k}N\\
\stackrel{\eqref{MainPfEq2}}\ll &\left(Sq_k\cdot\rho_{Sq_k}\Big(\frac{\delta N}{Sq_k}\Big)+Sq_k^{-1}\log^2q_k\right)^{\frac12}+ \delta+\frac {e^{\eta q_k}Sq_k}N.
\end{aligned}\end{equation}

Once $\delta$ is fixed, for sufficiently large $q_k$,
$$S q_k^{-1}\log^2 q_k\ll\delta^2.$$

Fix such a $q_k$ that, in addition, verifies $q_{k+1}>e^{\frac\tau2 q_k}$, then
$$
\bE_{n< N}\mu(n)f(T^n(x_0,y_0))
\ll \left(Sq_k\cdot\rho_{Sq_k}\Big(\frac{\delta N}{Sq_k}\Big)+\delta^2\right)^{\frac12}+ \delta+\frac {e^{\eta q_k}Sq_k}N.$$

Because $\delta$, $\eta$, $S$ and $q_k$ are now all fixed and $\rho_{Sq_k}$ is a function that decays to $0$, for suffciently large $N$,
$Sq_k\cdot\rho_{Sq_k}\Big(\frac{\delta N}{Sq_k}\Big)\ll\delta^2$ and $\frac {e^{\eta q_k}Sq_k}N\ll\delta$. In consequence, \eqref{DeltaEq} is verified for $N$ large enough.
As $\delta$ is arbitrary, \eqref{MainEq} holds, contradicting the hypothesis so far. This completes the proof of Theorem \ref{Main}.\end{proof}

\end{document}